\documentclass[12pt]{amsart}
\usepackage{amsmath,amsfonts,euscript,amscd,amsthm,amssymb,upref,graphics,color,verbatim}
\usepackage[all]{xy}
\usepackage[normalem]{ulem} 
\swapnumbers

\theoremstyle{plain}

\newtheorem{theorem}[subsection]{Theorem}
\newtheorem{nonameTHM}[subsection]{}
\newtheorem{proposition}[subsection]{Proposition}
\newtheorem{lemma}[subsection]{Lemma}
\newtheorem{corollary}[subsection]{Corollary}

\theoremstyle{definition}

\newtheorem{definition}[subsection]{Definition}
\newtheorem{definitions}[subsection]{Definitions}

\newtheorem{nothing*}[subsection]{}
\newtheorem{example}[subsection]{Example}

\newtheorem{notation}[subsection]{Notation}

\newtheorem{remark}[subsection]{Remark}

\newenvironment{enumerata}%
{\begin{enumerate}
		
		}{\end{enumerate}}
\newcommand{\Spec}{		\operatorname{{\rm Spec}}}
\newcommand{\Proj}{		\operatorname{{\rm Proj}}}

\newcommand{\trdeg}{	\operatorname{{\rm trdeg}}}
\newcommand{\Frac}{		\operatorname{{\rm Frac}}}
\newcommand{\HFrac}{		\operatorname{{\rm HFrac}}}
\newcommand{\ML}{		\operatorname{{\rm ML}}}

\newcommand{\lcm}{		\operatorname{{\rm lcm}}}
\newcommand{\type}{		\operatorname{{\rm type}}}
\newcommand{\cotype}{		\operatorname{{\rm cotype}}}

\newcommand{\Gr}{		\operatorname{{\rm\bf Gr}}}

\newcommand{\lb}{\langle}
\newcommand{\rb}{\rangle}

\newcommand{\setspec}[2]{\big\{\,#1\, \mid \,#2\, \big\}}

\newcommand{\Integ}{\ensuremath{\mathbb{Z}}}
\newcommand{\Nat}{\ensuremath{\mathbb{N}}}
\newcommand{\Rat}{\ensuremath{\mathbb{Q}}}
\newcommand{\Comp}{\ensuremath{\mathbb{C}}}

\newcommand{\proj}{\ensuremath{\mathbb{P}}}
\newcommand{\bk}{{\ensuremath{\rm \bf k}}}

\newcommand{\lnd}{\operatorname{{\rm LND}}}
\newcommand{\hlnd}{\operatorname{{\rm HLND}}}

\newcommand{\isom}{\cong}
\renewcommand{\epsilon}{\varepsilon}
\renewcommand{\phi}{\varphi}
\renewcommand{\emptyset}{\varnothing}

\CompileMatrices

\begin{document}
	\renewcommand{\baselinestretch}{1.07}

	%%%%%%	TOPMATTER:   %%%%%%%%%%%%%%%%%%%%%%%%%
	
	\title{On the rigidity of certain Pham-Brieskorn rings}

	\author{Michael Chitayat and Daniel Daigle}

\address{Department of Mathematics and Statistics\\
	University of Ottawa\\
	Ottawa, Canada\ \ K1N 6N5}

\email{mchit007@uottawa.ca}
\email{ddaigle@uottawa.ca}

\thanks{Research of both authors supported by grant 04539/RGPIN/2015 from NSERC Canada.}


{\renewcommand{\thefootnote}{}
}

\begin{abstract}
Fix a field $\bk$ of characteristic zero.
If $a_1, \dots, a_n$ ($n\ge3$) are positive integers,
the integral domain $B_{a_1, \dots, a_n} = \bk[X_1, \dots, X_n] / \langle X_1^{a_1} + \cdots +  X_n^{a_n} \rangle$ is called a \textit{Pham-Brieskorn ring}.
It is conjectured that if $a_i\ge2$ for all $i$ and $a_i=2$ for at most one $i$, then $B_{a_1, \dots, a_n}$ is rigid.
(A ring $B$ is said to be \textit{rigid} if the only locally nilpotent derivation $D: B \to B$ is the zero derivation.)
We give partial results towards the conjecture.
\end{abstract}
	
	\maketitle
	
	\vfuzz=2pt

\section{Introduction}

If $B$ is a commutative ring of characteristic zero, a derivation $D: B \to B$ is \textit{locally nilpotent} if for each $x \in B$ there exists $n>0$
such that $D^n(x)=0$. If the only locally nilpotent derivation $D: B \to B$ is the zero derivation, one says that $B$ is \textit{rigid}.
One says that $B$ is \textit{stably rigid} if for any $N\ge0$ and any locally nilpotent derivation $D : B[X_1, \dots, X_N] \to B[X_1, \dots, X_N]$
(where $B[X_1, \dots, X_N]$ is the polynomial ring in $N$ variables over $B$), we have $D(B) = \{0\}$.
Note that stable rigidity implies rigidity.

Fix a field $\bk$ of characteristic zero.
If $a_1, \dots, a_n$ ($n\ge3$) are positive integers,
the integral domain $B_{a_1, \dots, a_n} = \bk[X_1, \dots, X_n] / \langle X_1^{a_1} + \cdots +  X_n^{a_n} \rangle$ is called a \textit{Pham-Brieskorn ring}.
These rings and the corresponding varieties have been studied extensively and from several angles;
paper \cite{freudenburg2013} refers to \cite{SeadeBook2006} for a survey.
It is interesting to ask which Pham-Brieskorn rings are rigid or stably rigid.
Consider the set
$$
T_n = \setspec{ (a_1,\dots,a_n) \in \Integ^n }{ \text{$a_i \ge 2$ for all $i$ and $a_i=2$ for at most one $i$} } . 
$$
It is known (and easy to see) that 
\textit{if $B_{a_1,\dots,a_n}$ is rigid then $(a_1,\dots,a_n) \in T_n$};
so one wants to know if the converse is true. The case $n=3$ is settled by:

\begin{theorem}  \label {jeuyuBdgl3r6hj24dUi476dyb3i8}
Let $a,b,c$ be positive integers.
\begin{enumerata}

\item If $(a,b,c) \in T_3$ then $B_{a,b,c}$ is rigid.

\item If $\frac1a + \frac1b + \frac1c \le 1$ then  $B_{a,b,c}$ is stably rigid.

\end{enumerata}
\end{theorem}
This is \cite[Theorem 7.1]{freudenburg2013} (the case $\bk=\Comp$, with part (b) implicit, is \cite[Lemma 4]{Kali-Zaid_2000}).

For arbitrary $n\ge3$, one has:

\begin{theorem}\label {lowSum}
If $\sum\limits_{i=1}^n \frac{1}{a_i} \leq \frac{1}{n-2}$ then $B_{a_1,\dots,a_n}$ is rigid.
\end{theorem} 

The case $\bk=\Comp$ of Theorem~\ref{lowSum} is \cite[Example~2.6]{almostRigidRings}.
The general case follows, because (as one can see) rigidity over $\Comp$ implies rigidity over any field of characteristic zero.

In the case $n=4$, the following is known:

	\begin{theorem}  \label {collection}
Assume that $(a,b,c,d) \in T_4$.  Then $B_{a,b,c,d}$ is rigid in each of the following cases:
		
		\begin{enumerata}
			
		\item  $\frac{1}{a} + \frac{1}{b} + \frac{1}{c} + \frac{1}{d} \leq \frac{1}{2}$
		
		\item $\gcd(abc,d) = 1$
		
		\item $a=b=c=3$ 
		
		\item $a=2$, $b,c,d \geq 3$, $b$ is even, $\gcd(b,c) \geq 3$ and $\gcd(d,\lcm(b,c)) = 2$
		
		\item $\cotype(a,b,c,d) \geq 2$.
		
		\end{enumerata}
	\end{theorem}

Part (e) of Theorem \ref{collection} is Theorem 7.2(b) of \cite{LNDsAbelianGroup}, reformulated in terms of cotype (see Definition~\ref{typeDef} for the notion of cotype).
Part (a) is the case $n=4$ of Theorem~\ref{lowSum}.
Parts (b--d) are stated in \cite[Theorem 7.1]{LNDsAbelianGroup} but are proved in
\cite[Theorem~8.1]{freudenburg2013} and \cite[Corollary~1.9]{affineCones}.

\medskip

% \begin{itemize}
% 
% \item This article settles many cases not covered by Theorems \ref{jeuyuBdgl3r6hj24dUi476dyb3i8}--\ref{collection}.
% We gather some of our results in Theorem~\ref{cijvbb2i38dbfoqopdcwue} but several significant results from 
% Section~\ref{SEC:RigidityofPhamBrieskornRings} cannot be included in its statement
% because we want to avoid giving too many definitions in the Introduction.
% 
% \item This article settles many cases not covered by Theorems \ref{jeuyuBdgl3r6hj24dUi476dyb3i8}--\ref{collection}.
% In order to avoid giving too many definitions in the Introduction, we only present a subset of these cases as Theorem~\ref{cijvbb2i38dbfoqopdcwue}.
% Other significant results that are difficult to describe without the prerequisite definitions appear in Section~\ref{SEC:RigidityofPhamBrieskornRings}.

% \item

This article settles many cases not covered by Theorems \ref{jeuyuBdgl3r6hj24dUi476dyb3i8}--\ref{collection}.
In order to avoid giving too many definitions in the Introduction, we only present a subset of our results as Theorem~\ref{cijvbb2i38dbfoqopdcwue}.
Other significant results appear in Section~\ref{SEC:RigidityofPhamBrieskornRings}.
The Pham-Brieskorn rings $B_{a_1,\dots,a_n}$ that appear in Theorem~\ref{cijvbb2i38dbfoqopdcwue}
are defined over an arbitrary field $\bk$ of characteristic zero, and we always assume that $n\ge3$.

\begin{theorem} \label {cijvbb2i38dbfoqopdcwue}
\mbox{\ }

\begin{enumerata}
\setlength{\itemsep}{1.5mm}

\item If $a \ge n \ge 4$ then $B_{ \mbox{\scriptsize$\underbrace{a, \dots, a}_n$} }$ is rigid.
% Corollary~\ref{pcvn394rfwseEh39e}

\item If $\sum_{i=1}^n \frac1{a_i} \le \frac1{n-2}$ then $B_{a_1,\dots,a_n}$ is \textbf{stably} rigid.
% Corollary~\ref{ck9fweymrdkayecumcff83}

\item If $\sum\limits_{i \in I} \frac1{a_i} < \frac1{n-2}$ then $B_{a_1,\dots,a_n}$ is rigid,
where {\small $I = \setspec{ i }{ \text{$a_i$ divides $\lcm(a_1,\dots \widehat{a_i} \dots, a_n)$} }$.}
% Proposition~\ref{oneHalfRigid}

\item If $a,b,c,d \ge 1$ satisfy $a \nmid \lcm(b,c,d)$ and $\frac1b + \frac1c + \frac1d < \frac12$ then $B_{a,b,c,d}$ is rigid. 
% Corollary~\ref{0cvo3o90sduhvfTsaIHHFcJEr9}

\item If $(a_1,\dots,a_n) \in T_n$ and $\cotype(a_1,\dots,a_n) \geq n-2$, then $B_{a_1,\dots,a_n}$ is rigid.  
% Corollary~\ref{0cvikn3i4Apfsyuksukje}

\item If $k_1,k_2,k_3,k_4 \ge1$ are pairwise relatively prime and $a \ge 3$ then $B_{a k_1, a k_2, a k_3, a k_4}$ is rigid. 
% Corollary~\ref{cohj23r9fh03v8cjd3928}

\item If $k_1, \dots, k_n \ge1$ are pairwise relatively prime and $a \ge n \ge 4$ then $B_{a k_1, \dots, a k_n}$ is rigid. 
% Corollary~\ref{cohj23r9fh03v8cjd3928}

\item If $a_1,\dots,a_m\ge1$ $(m\ge1)$ satisfy $a_i \nmid \lcm(3,a_1, \dots \widehat{a_i} \dots, a_m)$ for all $i\in \{1, \dots, m\}$,
then $B_{a_1, \dots, a_m,3,3,3}$ is rigid.
% Example~\ref{0v34i8IeEkrj6ssyfuwygrwejk37rhf}
	
\end{enumerata}
\end{theorem}

Part~(a) (of Theorem~\ref{cijvbb2i38dbfoqopdcwue})
is Corollary~\ref{pcvn394rfwseEh39e}; the fact that  $B_{a,\dots,a}$ is rigid when $n \le a < n(n-2)$ appears to be a new result
(if $a\ge n(n-2)$ then $B_{a,\dots,a}$ is rigid by Theorem~\ref{lowSum}).
Parts (b) and (c) are Corollary~\ref{ck9fweymrdkayecumcff83} and Proposition~\ref{oneHalfRigid}, respectively;
these two results strengthen  Theorem~\ref{lowSum}.
Part (d) is the case ``$n=4$'' of part (c).
Part (e) is Corollary~\ref{0cvikn3i4Apfsyuksukje}; it generalizes Theorem~\ref{collection}(e).
Parts (f) and (g) are Corollary~\ref{cohj23r9fh03v8cjd3928}, and (h) is Example~\ref{0v34i8IeEkrj6ssyfuwygrwejk37rhf};
parts (f--h) are illustrations of stronger results that cannot be stated in the Introduction.

\bigskip

Let us say a few words about Section~\ref{SEC:DerivationsofGgradedrings}.
It is known that if a $\bk$-domain $B$ is not rigid then its field of fractions $\Frac(B)$ is ruled over $\bk$, i.e.,
there exists a field $K$ such that $\bk \subseteq K \subset \Frac(B)$ and $\Frac(B) = K(t)$ where $t$ is transcendental over $K$.
So, one technique for showing that $B$ is rigid is to show that $\Frac(B)$ is not ruled over $\bk$.
However, that technique is useless when $B$ admits a nontrivial $\Integ$-grading, because then $\Frac(B)$ is always ruled over $\bk$:
we have $\Frac(B) = K(t)$ where $K=\HFrac(B)$ is the ``homogeneous field of fractions'' of $B$.
Section~\ref{SEC:DerivationsofGgradedrings} shows that if $B$ is a graded $\bk$-domain which is not rigid then, under certain additional hypotheses, 
$\HFrac(B)$ itself is ruled over $\bk$.
So, to prove that a graded domain $B$ satisfying certain assumptions is rigid, it suffices to show that $\HFrac(B)$ is not ruled.
It is this technique that allows us to prove part (a) of Theorem~\ref{cijvbb2i38dbfoqopdcwue}.

\section{Preliminaries}
	
Throughout this work, all rings are commutative and have a multiplicative identity $1$.
All ring homomorphisms map 1 to 1. If $B$ is a ring then  $B^*$ denotes its group of units.
If $b \in B$ then $\lb b \rb$ is the ideal of $B$ generated by $b$,
and we use the notation $B_b = S^{-1}B$ where $S = \{1, b, b^2, \dots \}$.
By a \textit{domain}, we mean an integral domain.  If $B$ is a domain, its fraction field is denoted $\Frac(B)$.
If $A \subseteq B$ are domains then $\trdeg_A(B)$ denotes the transcendence degree of $\Frac(B)$ over $\Frac(A)$.
A subring $A$ of a domain $B$ is \textit{factorially closed in B} if for all $x,y \in B \setminus\{0\}$
we have the implication $xy \in A \Rightarrow x,y \in A$.
If $\bk$ is a field, then a \textit{$\bk$-domain} is a domain that is also a $\bk$-algebra.
An \textit{affine $\bk$-domain} is a $\bk$-domain that is finitely generated as a $\bk$-algebra.

If $A$ is a subring of a ring $B$, we write $B = A^{[n]}$ to indicate that $B$ is isomorphic to the polynomial algebra in $n$ variables over $A$.
If $K/\bk$ is a field extension, $K = \bk^{(n)}$ means that $K$ is purely transcendental over $\bk$, of transcendence degree $n$.

We write $``\subseteq"$ for inclusion and $``\subset"$ for proper inclusion.   
	
Let $B$ be a ring and $D : B \to B$ a derivation.
We say that $D$ is \textit{irreducible} if the only principal ideal of $B$ containing $D(B)$ is $B$ itself.      
We say that $D$ is \textit{locally nilpotent} if for each $b \in B$ there exists an $n \in \Nat$ such that $D^n(b) = 0$. 
A \textit{slice} of $D$ is an element $t \in B$ such that $D(t) = 1$.
A \textit{preslice} of $D$ is an element $t \in B$ such that $D(t) \neq 0$ and $D^2(t) = 0$.

\begin{definition}
If $B$ is a ring, the set of locally nilpotent derivations $D : B \to B$ is denoted $\lnd(B)$.
If $\lnd(B) = \{0\}$, we say that $B$ is \textit{rigid}.
\end{definition}

	\begin{nonameTHM} \label {p0cfi2k309cbqp90ws}
	{\rm Let $B$ be an integral domain of characteristic zero, let $D : B \to B$ be a derivation, and let $A = \ker D$.
The following facts are well known (refer to  \cite{VDE:book}, \cite{Freud:Book-new} or \cite{Dai:IntroLNDs2010}).}
		
		\begin{enumerata}

		\item \label {c0ovjn3vr7} If $D$ is locally nilpotent, then $A$ is a factorially closed subring of $B$.
Consequently, if $D$ is locally nilpotent and $\bk$ is a field included in $B$ then $D$ is a $\bk$-derivation.
		
		\item \label {teik5i68a9we} Assume that $\Rat \subseteq B$.
If $D \neq 0$ is locally nilpotent then $D$ has a preslice $t \in B$.
For any such $t$, if we define $\alpha = D(t)$ then $B_\alpha = A_\alpha[t] = (A_\alpha)^{[1]}$.
Consequently, $\trdeg_A(B) = 1$ and $\Frac(B)  = (\Frac(A))^{(1)}$.
		
		\item \label {02dj7edj9w34diey} If $D \neq 0$ is locally nilpotent and $B$ satisfies
the ascending chain condition for principal ideals, then there exists an irreducible
locally nilpotent derivation $\delta : B \to B$ such that $D = a \delta$ for some $a \in A$.  
		
		\item Let $S \subseteq B \setminus \{0\}$ be a multiplicative subset of $B$ containing $1$.
Then $S^{-1}D : S^{-1} B \to S^{-1} B$ defined by $(S^{-1}D)(\frac{b}{s}) = \frac{sD(b)-bD(s)}{s^2}$ is a derivation and the following hold:
		
		\begin{enumerata}
			\item $S^{-1}D$ is locally nilpotent if and only if $D$ is locally nilpotent and $S \subseteq A$;
		
			\item if $S \subseteq A$ then $\ker S^{-1}D = S^{-1}A$ and $S^{-1}A \cap B = A$.

		\end{enumerata}
	\end{enumerata}
	\end{nonameTHM}

\begin{definition} \label {iuh298grhgnfmdu7ueMc}
Let $B$ be a ring of characteristic zero.
If $D \in \lnd(B)$ then define $\deg_D : B \to \Nat \cup \{-\infty\}$ by declaring that $\deg_D(0)=-\infty$ and that
$\deg_D(x) = \max\setspec{ n \in \Nat }{ D^n(x) \neq 0 }$ for each $x \in B\setminus \{0\}$.
Also define $|x|_B = \min\setspec{ \deg_D(x) }{ D \in \lnd(B) \setminus \{0\} }$ for each $x \in B \setminus \{0\}$,
where we adopt the convention that $\min\emptyset = \infty$, so $|x|_B=\infty$ when $B$ is rigid.
\end{definition}
	
	\section{Derivations of G-graded rings}
\label {SEC:DerivationsofGgradedrings}

	\begin{definition}
		Let $(G,+)$ be an abelian group.
		A \textit{$G$-grading} of a ring $B$ is a family $\{B_g\}_{g \in G}$ of subgroups of $(B,+)$ such that
		$B = \bigoplus_{g \in G} B_g$
		and $B_g B_h \subseteq B_{g+h}$ for all $g, h \in G$.
		A \textit{$G$-graded ring} is a ring $B$ together with a $G$-grading (of $B$).
		In the special case where $G = \Integ$ and $B_i = 0$ for all $i < 0$,
		we say that $B$ is $\Nat$-graded and write $B = \bigoplus_{i \in \Nat} B_i$. 
	\end{definition}

	\begin{definitions}  \label {0j3w4c65mdry7u55gfyh6os}
		Suppose that $G$ is an abelian group and that  $B = \bigoplus_{i \in G} B_i$ is a $G$-graded ring. 
		
		\begin{itemize}
			
			\item A derivation $D : B \to B$ is \textit{homogeneous} if there exists an $h \in G$ such that $D(B_g) \subseteq B_{g+h}$ for all $g \in G$.
			If  $D$ is homogeneous and $D \neq 0$ then $h$ is unique and we say that $D$ is homogeneous of degree $h$.
			The zero derivation is homogeneous of degree $-\infty$.
			The set of homogeneous locally nilpotent derivations of $B$ is denoted $\hlnd(B)$. 
			
			\item A \textit{graded subring} of $B$ is a subring $A$ of $B$ satisfying  $A = \bigoplus_{g \in G} (A \cap B_g)$.
			If $A$ is a graded subring of $B$ then $A$ is a $G$-graded ring
			($A = \bigoplus_{g \in G} A_g$ is a $G$-grading of $A$, where we define $A_g = A \cap B_g$ for each $g \in G$).
			Note that if $D \in \hlnd(B)$ then $\ker D$ is a graded subring of $B$.
			
			\item If $A$ is a graded subring of $B$, define $G(A)$ to be the subgroup of $G$ generated by the set $\setspec{g \in G}{ A_g \neq 0 }$.

		\end{itemize}
	\end{definitions}
	
	\begin{definition}
		Suppose that $G$ is an abelian group and that  $B = \bigoplus_{i \in G} B_i$ is a $G$-graded integral domain.
		If $S$ is a multiplicatively closed subset of $\bigcup_{i \in G} ( B_i \setminus \{0\})$ such that $1 \in S$ then the localized ring
		$S^{-1}B$ is a $G$-graded ring in a natural way, and if we write $S^{-1}B = R = \bigoplus_{i \in G} R_i$ then
		the subring $R_0$ of $S^{-1}B$ called the \textit{homogeneous localization} of $B$ at $S$.
		Explicitly,
		$$
		R_0   % = \setspec{ \textstyle \frac bs }{ \text{for some $i \in G$ we have $b \in B_i$ and $s \in B_i \cap S$} }
= \bigcup_{i \in G} \setspec{ \textstyle \frac bs }{ \text{$b \in B_i$ and $s \in B_i \cap S$} } .
		$$
		We define $\HFrac(B)$ to be the homogeneous localization of $B$ at $S = \bigcup_{i \in G} ( B_i \setminus \{0\})$.
		It is clear that $\HFrac(B)$ is a subfield of $\Frac(B)$;
		we call $\HFrac(B)$ the \textit{homogeneous field of fractions} of $B$.
	\end{definition}

	\begin{example} \label {HFracExample}
		If $B$ is an $\Nat$-graded integral domain such that $B \neq B_0$ then $\Proj B$ and $\Spec B$ are integral schemes,
		$\Frac(B)$ is the function field of $\Spec B$, $\HFrac(B)$ is the function field of $\Proj B$,
		and $\Frac(B) = (\HFrac(B))^{(1)}$.
(If $B = B_0$ then $\Proj B = \emptyset$ is not an integral scheme and hence does not have a function field.
Note that $\Frac(B) = \HFrac(B)$ whenever $B=B_0$.)
	\end{example}

	\begin{definition}  \label {Cc0vnv3949eWv23f0w}
		A field extension $L / \bk$ is \textit{ruled} if there exists a field $K$ such that $\bk \leq K \leq L$ and $L = K^{(1)}$.  
	\end{definition}

\begin{theorem} \label {mainResultGeneralized} % \label {mainResult}
Let $G$ be an abelian group and $B$ a $G$-graded integral domain of characteristic zero.
Let $D \in \hlnd(B) \setminus \{0\}$, let $A = \ker D$ and suppose that $G(A) = G(B)$.
\begin{enumerata}

\item $\HFrac(B) = (\HFrac (A))^{(1)}$

\item If $\bk$ is a field included in $B$ then $\bk \cap B_0$ is a field included in $\HFrac(A)$.

\item If $G$ is torsion-free and $\bk$ is a field included in $B$ then $\bk \subseteq \HFrac(A)$ and consequently $\HFrac(B)$ is ruled over $\bk$.

\end{enumerata}
\end{theorem}
			
\begin{proof}
If $B=B_0$ then $\Frac(B) = \HFrac(B)$ and $\Frac(A) = \HFrac(A)$, so the Theorem follows from parts \eqref{c0ovjn3vr7} and \eqref{teik5i68a9we}
of \ref{p0cfi2k309cbqp90ws}.  So we may assume that $B \neq B_0$.

(a) Let $d = \deg D$ and $S = \bigcup_{g \in G} ( A_g \setminus \{0\})$.
Then $S^{-1}D : S^{-1}B \to S^{-1}B$ is nonzero and locally nilpotent, and also homogeneous of degree $d$.
Since $-d \in G(B)$ and $G(A) = G(B)$, there exist
$a,s \in S$ such that the homogeneous element $\frac as \in S^{-1}A$ has degree $-d$.
Then $\frac{a}{s} S^{-1}D : S^{-1}B \to S^{-1}B$ is a homogeneous derivation of degree $0$;
since $\frac{a}{s} \in \ker(S^{-1}D)$ and $S^{-1}D$ is locally nilpotent, $\frac{a}{s} S^{-1}D$ is locally nilpotent.
Let $B_{(S)}$ denote the degree 0 subring of $S^{-1}B$ and let $D_{(S)} : B_{(S)} \to B_{(S)}$ be the restriction of  $\frac{a}{s} S^{-1}D$.
Then $D_{(S)}$ is a locally nilpotent derivation whose kernel is $B_{(S)} \cap \ker\big( \frac{a}{s} S^{-1}D \big) = B_{(S)} \cap S^{-1}A = A_{(S)} = \HFrac(A)$.

We show that $D_{(S)}$ is nonzero. 
It is straightforward to verify that $D$ has a homogeneous preslice $t \in B$.
Since $G(A) = G(B)$, there exist $a',s' \in S$ such that $\deg(\frac{a'}{s'}) = -\deg(t)$.
Then $\frac{a't}{s'} \in B_{(S)}$ and, since $\frac{a'}{s'} \in \ker S^{-1}D$,
$D_{(S)}(\frac{a't}{s'}) =  \big( \frac{a}{s} S^{-1}D \big)(\frac{a't}{s'}) = \frac{a'}{s'} \frac{a}{s} S^{-1}D (t) \neq 0$.
		
Since $D_{(S)} \neq 0$ and $\ker D_{(S)} = \HFrac(A)$ is a field,
\ref{p0cfi2k309cbqp90ws}\eqref{teik5i68a9we} implies that $B_{(S)} = ( \HFrac (A) )^{[1]}$.
Using $G(A)=G(B)$ once again, we see that $\HFrac(B)$ is  the field of fractions of $B_{(S)}$; so $\HFrac(B) = (\HFrac (A))^{(1)}$.
This proves (a).

(b) If $\bk$ is a field included in $B$ then it is clear that $\bk \cap B_0$ is a field
and \ref{p0cfi2k309cbqp90ws}\eqref{c0ovjn3vr7} implies that $\bk \subseteq A$, so $\bk \cap B_0 \subseteq A \cap B_0 = A_0 \subseteq \HFrac(A)$. 

(c) It is well known that {\it if $G$ is a torsion-free abelian group and
$B$ is a $G$-graded domain then any field included in $B$ is in fact included in $B_0$}
(see \cite[Lemma 2.4.7]{ChitayatMScThesis}, for instance).
So in the present situation we have $\bk \subseteq B_0$; then (b) implies that $\bk \subseteq \HFrac(A)$ and (a) implies that $\HFrac(B)$ is ruled over $\bk$.
\end{proof}

\begin{remark}  \label {c09vb1923dchsDd}
In the special case where the grading is an $\Nat$-grading and is nontrivial ($B \neq B_0$),
Theorem~\ref{mainResultGeneralized}(c) asserts that the function field of $\Proj B$ is ruled over $\bk$.
(See Example~\ref{HFracExample}.) The same remark applies to Proposition~\ref{ckjv2p39wbpoqQds}, below.
\end{remark}

\begin{nothing*}[Homogenization] \label {cpijbo329eufds} 
Let $\bk$ be a field of characteristic zero and $B = \bigoplus_{i \in \Integ} B_i$ a $\Integ$-graded affine $\bk$-domain.
One can show that if $D \in \lnd(B)$ then the subset $S_D \overset{\text{\tiny\rm def}}{=} \setspec{ \deg( Dx ) - \deg(x) }{ x \in B \setminus \{0\} }$ 
of $\Integ \cup \{-\infty\}$ has a greatest element (where $\deg:B\to\Integ\cup\{-\infty\}$ is the degree function determined by the grading of $B$);
one defines $\deg(D) = \max(S_D) \in \Integ \cup \{-\infty\}$.
Clearly, $\deg(D)=-\infty$ $\Leftrightarrow$ $D=0$.
Also, if $D$ happens to be homogeneous then $\deg(D)$ coincides with the usual degree of a homogeneous derivation (Definition~\ref{0j3w4c65mdry7u55gfyh6os}).

If $D \in \lnd(B)$ then there is a natural way to define an element $\tilde D$ of $\hlnd(B)$ satisfying (in particular) $\deg(\tilde D) = \deg(D)$.
Indeed, if $D=0$ then set $\tilde D = 0$.
If $D \in \lnd(B) \setminus \{0\}$ then let $d = \deg(D) \in \Integ$ and define (for each $i \in \Integ$) a map $\tilde D_i : B_i \to B_{i+d}$
by $\tilde D_i (x) = p_{i+d}( D x )$, where $p_j : B \to B_j$ is the canonical projection;
then if $b = \sum_{i \in \Integ} b_i \in B$ ($b_i \in B_i$ for all $i$, $b_i=0$ for almost all $i$), define $\tilde D(b) = \sum_i \tilde D_i(b_i)$;
one can check that $\tilde D \in \hlnd(B) \setminus \{0\}$ and  $\deg(\tilde D) = \deg(D)$. Note in particular that
$D \neq 0 \implies \tilde D \neq 0$.
The derivation $\tilde D$ is sometimes called the \textit{homogenization} of $D$.
This is in fact a special case of the process of replacing $D : B \to B$ by $\Gr(D) : \Gr(B) \to \Gr(B)$,
so one can also call  $\tilde D$ the \textit{associated homogeneous derivation} of $D$.
Note the following consequence of the above discussion:
\begin{quote}
\it
Let $\bk$ be a field of characteristic zero and $B$ a $\Integ$-graded affine $\bk$-domain.
If $B$ is not rigid then $\hlnd(B) \neq \{0\}$.
\end{quote}
Refer to \cite{Dai:TameWild} for proofs of the claims made in \ref{cpijbo329eufds}, and for an in-depth treatment of this topic.
\end{nothing*}

\begin{definitions}	\label {typeDef}
Let $n\ge2$ and $S = (a_1, a_2,  \dots,  a_n) \in \Integ^n$.  
\begin{itemize}

\item Define\footnote{By convention, $\gcd(S)\ge0$ and $\lcm(S)\ge0$.}  $\gcd(S) = \gcd(a_1, \dots, a_n)$ and $\lcm(S) = \lcm(a_1,  \dots,  a_n)$.

\item If $\gcd(S) = 1$, we say that $S$ is \textit{normal}.
If $S \neq (0, \dots, 0)$ then the tuple $S' = (\frac{a_1}{d},  \dots,   \frac{a_n}{d})$ (where $d = \gcd(S)$) is normal,
and is called the \textit{normalization of $S$}.

\item For each $j \in \{1,\dots,n\}$, define $S_j = (a_1, \dots \, \widehat{a_j} \, \dots, a_n)$ ($j$-th component omitted).
More generally, given a proper subset $J$ of $\{1,\dots,n\}$ we define $S_J = (a_{i_1}, \dots, a_{i_s})$, where\\
$i_1 < \cdots < i_s$ are the elements of $\{1, \dots, n\} \setminus J$.
			
\item We define the sets
\begin{align*}
J^*(S) &= \setspec{i \in \{ 1, \dots, n \} }{ \gcd(S_i) \neq \gcd(S) } = \setspec{i \in \{ 1, \dots, n \} }{ \gcd(S_i) \nmid a_i } \\
J(S) &= \setspec{i \in \{ 1, \dots, n \} }{ \lcm(S_i) \neq \lcm(S) }  = \setspec{i \in \{ 1, \dots, n \} }{ a_i \nmid \lcm(S_i) }  .
\end{align*}
Then we define the \textit{type} and the \textit{cotype} of $S$ by:
$$
\type(S) = | J^*(S) | \quad \text{and} \quad \cotype(S) = | J(S) | .
$$
Note that $\type(S), \cotype(S) \in \{0,1,\dots,n\}$ and that, if $S'$ is the normalization of $S$,
then $\type(S) = \type(S')$ and $\cotype(S) = \cotype(S')$.

\end{itemize}  
\end{definitions}

\begin{proposition}  \label {ckjv2p39wbpoqQds}
Let $\bk$ be a field of characteristic zero and $B$ a $\Integ$-graded $\bk$-domain.
Suppose that there exist homogeneous prime elements $x_1, \dots, x_n$ $(n\ge2)$ of $B$ satisfying the following
conditions, where $d_i = \deg x_i$ for all $i$:
\begin{itemize}

\item $\langle d_1, \dots, d_n \rangle = \Integ(B)$ and  $\type( d_1, \dots, d_n ) = 0$

\item  for any choice of distinct $i,j$, the elements $x_i, x_j$ are not associates in $B$.

\end{itemize}
Then the following hold.
\begin{enumerata}

\item $\Integ(\ker D) = \Integ(B)$ for all $D  \in \hlnd(B)$.

\item If $B$ is $\bk$-affine\footnote{We mean that $B$ is finitely generated as a $\bk$-algebra.}  and not rigid then $\HFrac(B)$ is ruled over $\bk$.

\end{enumerata}
\end{proposition}

\begin{proof}
The assumption $\langle d_1, \dots, d_n \rangle = \Integ(B)$ and  $\type( d_1, \dots, d_n ) = 0$
is equivalent to:
\begin{center}
for each subset $I \subset \{1, \dots, n\}$ of cardinality $n - 1$, $\setspec{ \deg(x_i) }{ i \in I }$ generates $\Integ(B)$.
\end{center}
So assertion (a) is the special case $G=\Integ$ of \cite[Corollary~4.2]{LNDsAbelianGroup}.
To prove (b), suppose that $B$ is $\bk$-affine and not rigid; then \ref{cpijbo329eufds} implies that $\hlnd(B) \setminus \{0\} \neq \emptyset$.
Pick $D \in \hlnd(B) \setminus \{0\}$, then $\Integ( \ker D ) = \Integ(B)$ by part (a),
so Theorem~\ref{mainResultGeneralized}(c) implies that  $\HFrac(B)$ is ruled over $\bk$.
\end{proof}

	\section{Rigidity of Pham-Brieskorn Rings}
\label {SEC:RigidityofPhamBrieskornRings}
	
Let $\bk$ be a field of characteristic zero.

\begin{definition} \label {8vroq98vbcu56msw9nw0e} 
Given $n\ge3$ and $S = (a_1, \dots, a_n) \in (\Nat \setminus \{0\})^n$, we use the notation
$$
B_S = B_{a_1, \dots, a_n} = \bk[X_1, \dots ,X_n]/ \lb X_1^{a_1} +  \cdots  + X_n^{a_n} \rb .
$$
The $\bk$-domain $B_{a_1, \dots, a_n}$ is called a \textit{Pham-Brieskorn ring}.
We use the capital letters $X_i$ to represent the variables in the polynomial ring,
and define $x_i = \pi(X_i) \in B_S$ where $\pi : \bk[X_1, \dots ,X_n] \to B_S$ is the canonical quotient map.
Thus $B_S = \bk[x_1, \dots, x_n]$.
Let $(d_1,\dots,d_n) =  \big( \frac L{a_1}, \dots,  \frac L{a_n} \big)$, where $L = \lcm(S)$;  
then there is a unique $\Nat$-grading of $B_S$ with the property that $x_i$ is homogeneous of degree $d_i$ for all $i=1,\dots,n$.
We call it the {\it standard $\Nat$-grading} of $B_S$.
\end{definition}

\begin{lemma} \label {equivalentType2}
Given $n\ge2$ and $S = (a_1,\dots,a_n) \in (\Nat \setminus \{0\})^n$,
define $\bar{S} = \big( \frac L{a_1}, \dots,  \frac L{a_n} \big)$ where $L = \lcm(S)$.
Then the following hold.
\begin{enumerata}

\item $\bar S$ is normal and $\overline{ (\bar S) }$ is the normalization of $S$.

\item $J^*( \bar S ) = J(S)$\ \  and\ \ $J^*(S) = J(\bar S)$.

\item  $\type(\bar S) = \cotype( S )$\ \ and\ \ $\type(S) = \cotype( \bar S )$.

\end{enumerata}
\end{lemma}

The proof of Lemma~\ref{equivalentType2} is left to the reader.
We deduce the following useful triviality:

\begin{lemma}  \label {pc04b98wjwk77subrq9a}
Let $n \ge4$ and $S = (a_1, \dots, a_n) \in (\Nat \setminus \{0\})^n$, and consider $B_S$,
$x_1, \dots, x_n$ and $(d_1, \dots, d_n)$ as in Definition~\ref{8vroq98vbcu56msw9nw0e}.
\begin{enumerata}

\item $\gcd(d_1,\dots,d_n) = 1$\ \  and\ \ $\type(d_1,\dots,d_n) = \cotype(a_1,\dots,a_n)$

\item For each $i \in \{1, \dots, n\}$, we have
$$
\gcd(d_1, \dots \widehat{d_i} \dots, d_n) \neq 1 \iff \lcm(a_1, \dots \widehat{a_i} \dots, a_n) \neq \lcm(S) \iff i \in J(S) .
$$

\item $x_1, \dots, x_n$ are homogeneous prime elements of $B_S$ and are pairwise non-associates.

\end{enumerata}
\end{lemma}

\begin{proof}
Since $(d_1,\dots,d_n) = \bar S$, assertions (a) and (b) follow from Lemma~\ref{equivalentType2}.
For each $i$,
define $S_i$ as in Definition~\ref{typeDef}; then $B_S / \langle x_i \rangle \isom B_{S_i}$, which is a domain because $n \ge 4$.
So $x_1, \dots, x_n$ are homogeneous prime elements of $B_S$, and it is clear that they are pairwise non-associates.
\end{proof}

Recall that a $\bk$-variety $X$ is said to be \textit{ruled} if the function field of $X$ is ruled over $\bk$,
in the sense of Definition~\ref{Cc0vnv3949eWv23f0w}.

\begin{theorem}  \label {dp0c9vn239efps0}
Let $n\ge4$, and let $S \in (\Nat \setminus \{0\})^n$ be such that $\cotype(S)=0$.
\begin{enumerata}

\item We have $\Integ( \ker D ) = \Integ$ for all $D \in \hlnd(B_S)$.

\item If $B_S$ is not rigid, then the projective $\bk$-variety $\Proj B_S$ is ruled over $\bk$.

\end{enumerata}
\end{theorem}

\begin{proof}
By Example~\ref{HFracExample}, the function field of $\Proj B_S$ is $\HFrac(B_S)$.
Consider $x_1, \dots, x_n$ and $(d_1, \dots, d_n)$ as in Definition~\ref{8vroq98vbcu56msw9nw0e}.
By Lemma \ref{pc04b98wjwk77subrq9a},  $x_1, \dots, x_n$  satisfy the hypothesis of  Proposition~\ref{ckjv2p39wbpoqQds}
(in particular $\type(d_1,\dots,d_n)= \cotype(S)=0$).
So Proposition~\ref{ckjv2p39wbpoqQds} implies that (a) and (b) hold.
\end{proof}

\begin{remark} \label {c09238effgCq26}
If $a \ge n \ge 4$ then the Fermat variety
$F_{a,n} = \Proj\big( \bk[X_1,\dots,X_n] / \langle X_1^a + \cdots + X_n^a \rangle \big)$
is not uniruled, hence not ruled.
(Over a field of characteristic zero,
a smooth hypersurface of degree $d$ in $\proj^N$ is uniruled if and only if $d \le N$.)
\end{remark}

\begin{corollary}  \label {pcvn394rfwseEh39e}
If $S = (\underbrace{a, \dots, a}_n)$ satisfies $a \ge n \ge 4$, then $B_S$ is rigid.
\end{corollary}

\begin{proof}
Note that $\cotype(S)=0$.
If $B_S$ is not rigid then (by Theorem~\ref{dp0c9vn239efps0}) $\Proj B_S = F_{a,n}$ is ruled, which is not the case by Remark~\ref{c09238effgCq26}.
\end{proof}

We shall now develop a different approach for proving that $B_S$ is rigid in certain cases.
We need the following result, which is part (b) of Corollary 3.3 of \cite{freudenburg2013}.
See \ref{iuh298grhgnfmdu7ueMc} for the definition of $|u|_A$.
% See Definition \ref{iuh298grhgnfmdu7ueMc} for the definition of $|u|_A$.

\begin{corollary} \label {rigidDegree2}
Suppose $R$ is a $\Integ$-graded affine $\bk$-domain, $f \in R$ is homogeneous, 
and $n \geq 2$ is an integer not dividing $\deg f$. Set $g = \gcd(n, \deg f)$, define the rings
$$
A = R[u] / \lb f + u^g \rb \text{\ \ and\ \ } B = R[z] / \lb f + z^n \rb
$$
and assume that $B$ is a domain.  Then $A$ is a domain and
$$
\text{$B$ is rigid if and only if  $|u|_A \geq 2$.}
$$
\end{corollary}

\begin{definition} \label{partialOrdering}
Let $n\geq3$.
\begin{enumerata}

\item Given $S = (a_1,\dots,a_n) \in (\Nat\setminus\{0\})^n$ and $i \in \{1, \dots, n\}$, define $g_i(S) = \gcd(a_i, \lcm(S_i))$ (recalling the definition of $S_i$ from Definition \ref{typeDef}).

\item Let $S = (a_1,\dots,a_n)$ and $S' = (a_1',\dots,a_n')$ be elements of $(\Nat\setminus\{0\})^n$ and let $i \in \{1, \dots, n\}$.
We write $S \le^i S'$ if and only if
$$
S_i = S'_i \quad \text{and} \quad g_i(S') \mid a_i \mid a_i' .
$$
We write $S <^i S'$ if and only if $S \le^i S'$ and $S \neq S'$.
\end{enumerata}
Observe that $\le^i$ is a partial order on $(\Nat\setminus\{0\})^n$
(transitivity follows from the fact that if $S \le^i S'$ then $g_i(S) = g_i(S')$).
\end{definition}

\begin{proposition} \label {cljbo238ecv0}
Let $n\ge3$.
\begin{enumerata}

\item Suppose that $S^*, S \in (\Nat\setminus\{0\})^n$  and $i \in \{1,\dots,n\}$ satisfy $S^* \le^i S$.  If $B_{S^*}$ is rigid then so is $B_{S}$.

\item Let $S, S' \in (\Nat\setminus\{0\})^n$, and suppose that there exist $i \in \{1,\dots,n\}$ and
$S^* \in (\Nat\setminus\{0\})^n$ such that $S^* <^i S$ and $S^* <^i S'$. 
Then $B_S$ is rigid if and only if $B_{S'}$ is rigid.

\end{enumerata}
\end{proposition}

\begin{proof}
(b) We may assume that $i=n$.
Consider $S = (a_1,\dots,a_n)$, $S' = (a_1',\dots,a_n')$ and $S^* = (a_1^*,\dots,a_n^*)$ satisfying $S^* <^n S$ and $S^* <^n S'$.
The hypothesis implies that 
$$
S_n = S_n^* = S_n', \quad g \mid a_n^* \mid a_n \quad \text{and} \quad  g \mid a_n^* \mid a_n' \, , 
$$
where $g = g_n(S) = g_n(S^*) = g_n(S')$.
Let $L = \lcm(a_1, \dots, a_{n-1})$ and define an $\Nat$-grading on $R = \bk[X_1,\dots,X_{n-1}]$ by declaring that (for each $i=1,\dots,n-1$)
$X_i$ is homogeneous of degree $\frac{a_n^*}{g} \frac{L}{a_i}$. Then $f = X_1^{a_1} + \cdots +X_{n-1}^{a_{n-1}}$ is homogeneous of degree $\frac{a_n^* L}{g}$.
Consider the rings
$$
A = R[u]/ \lb f + u^{a_n^*} \rb, \quad B = R[z]/ \lb f + z^{a_n} \rb, \quad B' = R[w]/ \lb f + w^{a_n'} \rb . 
$$
Observe that $A \isom B_{S^*}$, $B \isom B_{S}$ and $B' \isom B_{S'}$.
The reader may verify that $a_n \nmid \deg f$,  $a_n' \nmid \deg f$ and $\gcd(a_n, \deg f) = a_n^* = \gcd(a_n', \deg f)$;
thus each one of the pairs $(A,B)$, $(A,B')$ satisfies the hypothesis of Corollary~\ref{rigidDegree2}.
By that result, $B$ is rigid $\Leftrightarrow$ $|u|_A \ge 2$ $\Leftrightarrow$ $B'$ is rigid, so (b) is proved.
It also follows that (a) is true, because if $A$ is rigid then $|u|_A = \infty \ge 2$, so $B$ is rigid.
\end{proof}

\begin{corollary} \label {cohj23r9fh03v8cjd3928}
\mbox{\ }
\begin{enumerata}

\item If $k_1,k_2,k_3,k_4$ are pairwise relatively prime positive integers and $a \ge 3$ then $B_{a k_1, a k_2, a k_3, a k_4}$ is rigid. 

\item If $k_1, \dots, k_n$ are pairwise relatively prime positive integers
and $a \ge n \ge 4$ then $B_{a k_1, \dots, a k_n}$ is rigid. 

\end{enumerata}
\end{corollary}

\begin{proof}
Define $S^0 = (a,a,a,a)$ to prove (a), and $S^0 = (\overbrace{a,\dots,a}^n)$ to prove (b).
Note that $B_{S^0}$ is rigid  (by Theorem~\ref{collection}(c) if $S^0=(3,3,3,3)$, by Corollary~\ref{pcvn394rfwseEh39e} in the other cases).
Define 
$S^1 = (ak_1,a, a, \dots )$, $S^2 = (ak_1,ak_2, a,\dots)$, \dots, $S^{n-1} = (ak_1,\dots,ak_{n-1},a)$, $S^n = (ak_1,\dots,ak_n)$;
then $S^0 \le^1 S^1 \le^2 S^2  \le^3 \cdots \le^{n-1} S^{n-1}  \le^n S^n$.
Since $B_{S^0}$ is rigid, so is  $B_{S^n}$ by Proposition~\ref{cljbo238ecv0}.
\end{proof}

\begin{lemma} \label {lcjcbv924yowef}
Given $n \ge3$ and $S = (a_1,\dots,a_n) \in (\Nat\setminus\{0\})^n$,
$$
J(S) = \setspec{ j }{ \text{$S$ is not minimal with respect to $<^j$} }.
$$
\end{lemma}

\begin{proof}
If $j \in J(S)$ then $a_j \nmid \lcm(S_j)$, so $g_j(S) \neq a_j$ and hence 
$$
(a_1, \dots, a_{j-1}, g_j(S), a_{j+1}, \dots, a_n) <^j S .
$$
Conversely, if $(a_1',\dots,a_n') <^j S$ then $g_j(S) \mid a_j' \mid a_j$ and $a_j'\neq a_j$, so $j \in J(S)$.
\end{proof}

\begin{proposition} \label {oneHalfRigid}
Let $n \ge3$, $S = (a_1,\dots,a_n) \in (\Nat\setminus\{0\})^n$ and 
$$
I(S) = \{1, \dots, n\} \setminus J(S) = \setspec{ i }{ \text{$a_i$ divides $\lcm(S_i)$} } . 
$$
If $\sum\limits_{i \in I(S)} \frac1{a_i} < \frac1{n-2}$ then $B_S$ is rigid.
\end{proposition}

\begin{proof}
Let us write $I = I(S)$ and $J = J(S)$.
If $J = \emptyset$ then $\sum\limits_{i = 1}^n \frac1{a_i} < \frac1{n-2}$ so we are done by Theorem~\ref{lowSum}.
Assume that $J \neq \emptyset$ and let $j_1 < \cdots < j_k$ be the elements of $J$.
Choose distinct primes $p_{j_1}, \dots, p_{j_k}$ such that $\gcd(\prod_{j \in J} p_j ,\, \prod_{i=1}^n a_i)=1$.
Choose positive integers $e_{j_1}, \dots, e_{j_k}$ such that 
$\sum\limits_{i \in I} \frac1{a_i} + \sum\limits_{j \in J} \frac1{a_{j} p_{j}^{e_j}} < \frac1{n-2}$.
We inductively define a sequence $S^0, S^1, \dots, S^k$ of elements of $(\Nat\setminus\{0\})^n$ by setting $S^0 = S$ and, for each $\nu=1,\dots,k$,
$$
S^\nu = \text{componentwise product } (1, \dots, 1, p_{j_\nu}^{ e_{j_\nu} }, 1, \dots, 1) \cdot S^{\nu-1} 
$$
where the $p_{j_\nu}^{ e_{j_\nu} }$ is in the $j_\nu$-th position.
The reader can check that $J(S^\nu) = J(S) = J$ for all $\nu$ and that 
$S^0 \le^{j_1} S^1 \le^{j_2} S^2 \le^{j_3} \cdots \le^{j_k} S^k$.

Let $\nu \in \{1, \dots, k\}$;
since $j_\nu \in J = J(S^{\nu-1})$, Lemma \ref{lcjcbv924yowef} implies that $S^{\nu-1}$ is not minimal
with respect to $\le^{j_\nu}$, i.e., there exists $S^{\nu-1}_* \in (\Nat\setminus\{0\})^n$ satisfying  $S^{\nu-1}_* <^{j_\nu} S^{\nu-1}<^{j_\nu} S^\nu$;
then Proposition~\ref{cljbo238ecv0}(b) implies that $B_{S^{\nu-1}}$ is rigid if and only if $B_{S^{\nu}}$ is rigid.
As this holds for all $\nu \in \{1, \dots, k\}$, $B_S=B_{S^0}$ is rigid  if and only if $B_{S^{k}}$ is rigid.
Now $B_{S^{k}}$ is indeed rigid by Theorem \ref{lowSum}, because if we write $S^k = (a_1',\dots,a_n')$ then
$\sum\limits_{i =1}^n \frac1{a_i'} = \sum\limits_{i \in I} \frac1{a_i} + \sum\limits_{j \in J} \frac1{a_{j} p_{j}^{e_j}} < \frac1{n-2}$.
So $B_S$ is rigid.
\end{proof}

Proposition~\ref{oneHalfRigid} shows that many Pham-Brieskorn varieties that are not shown to be rigid
by Theorems \ref{jeuyuBdgl3r6hj24dUi476dyb3i8}--\ref{collection} are indeed rigid.
In the case $n=4$, the following is an immediate consequence of Proposition~\ref{oneHalfRigid}:

\begin{corollary}  \label {0cvo3o90sduhvfTsaIHHFcJEr9}
If $a,b,c,d \in \Nat\setminus\{0\}$ satisfy $a \nmid \lcm(b,c,d)$ and $\frac1b + \frac1c + \frac1d < \frac12$ then $B_{a,b,c,d}$ is rigid. 
\end{corollary}

We need the following:

\begin{corollary}  \label {Cor63ofDFM}
Let $n\ge4$ and $S \in (\Nat\setminus\{0\})^n$. 
Consider $B_S = \bk[x_1, \dots, x_n]$ with $x_1, \dots, x_n$ as in Definition~\ref{8vroq98vbcu56msw9nw0e}. 
Then for every $D \in \hlnd(B_S)$ the following conditions hold.
\begin{enumerata}

\item  If $j \in J(S)$ then $D^2(x_j)=0$.

\item If $j_1,j_2$ are distinct elements of $J(S)$ then $D(x_{j_1})=0$ or $D(x_{j_2})=0$.

\end{enumerata}
\end{corollary}

\begin{proof}
In view of Lemma~\ref{pc04b98wjwk77subrq9a}(b), this is a special case of Corollary 6.3 of \cite{LNDsAbelianGroup}.
\end{proof}

\begin{proposition}  \label {kcj203jfs0dvj}
Let $n\ge4$ and $S \in (\Nat\setminus\{0\})^n$.  Assume that $J(S) \neq \emptyset$ and that
\begin{equation} \tag{$*$}
\text{$B_{S_J}$ is rigid for every subset $J$ of $J(S)$ satisfying $|J| = \min( |J(S)|-1, n-3 )$.}
\end{equation}
Then $B_S$ is rigid.
\end{proposition}

\begin{proof}
If  $| J(S) | = 1$ then $\min( |J(S)|-1, n-3 )=0$, so $(*)$ reads ``$B_S$ is rigid,'' 
so the Proposition is trivially true in this case.
From now-on, we assume that $| J(S) | > 1$.
By contradiction, assume that $B_S$ is not rigid. 
Then $\hlnd(B_S) \neq \{0\}$ by \ref{cpijbo329eufds}; choose an irreducible $D \in \hlnd(B_S) \setminus \{0\}$
(this is possible by \ref{p0cfi2k309cbqp90ws}\eqref{02dj7edj9w34diey}).
By Corollary~\ref{Cor63ofDFM}, we have $D(x_{j_1}) = 0$ or $D(x_{j_2}) = 0$ for every choice of distinct $j_1,j_2 \in J(S)$.
So there exists a subset $J'$ of $J(S)$ such that  $|J'| = |J(S)|-1$ and $D(x_j)=0$ for all $j \in J'$;
hence there exists a $J \subseteq J' \subset J(S)$ such that  $|J| = \min( |J(S)|-1, n-3 )$ and $D(x_j)=0$ for all $j \in J$.
We might as well assume that $J = \{ m+1, m+2, \dots, n\}$ for some $m\ge3$.
For each $i$ such that $m \le i \le n$ we define $P(i)$ to be the statement
\begin{quote}
there exists an irreducible $D_i \in \hlnd( B_{a_1, \dots, a_i} ) \setminus \{0\}$\\
satisfying $D_i( x_j )=0$ for all $j \in \{ 1, \dots, i \} \cap J$
\end{quote}
where $x_j$ denotes the image of $X_j$ in $B_{a_1, \dots, a_i}$.
Then $P(n)$ is true, with $D_n=D$.
Assume that $i$ is such that $m<i\le n$ and such that $P(i)$ is true.
Note that $i \in J$, so $D_i(x_i)=0$, so $D_i$
induces a homogeneous locally nilpotent derivation $\bar D_i$ of the ring
$B_{a_1, \dots, a_{i}}/ \langle x_i \rangle \isom B_{a_1, \dots, a_{i-1}}$,
and $\bar D_i \neq 0$ because $D_i$ is irreducible.
Moreover, $\bar D_i( \bar x_j ) = 0$ for all $j \in  \{ 1, \dots, i-1 \} \cap J$.
By \ref{p0cfi2k309cbqp90ws}\eqref{02dj7edj9w34diey}, there exists an irreducible $D_{i-1} \in \hlnd( B_{a_1, \dots, a_{i-1}} ) \setminus \{0\}$
such that $\ker D_{i-1} = \ker \bar D_i$, so $P(i-1)$ is true.
By descending induction, $P(n), P(n-1), \dots, P(m)$ are all true.
Since $P(m)$ is true and $S_J=(a_1,\dots,a_m)$, we have $D_m \in \hlnd( B_{S_J} ) \setminus \{0\}$, so $B_{S_J}$ is not rigid, contradicting $(*)$.
\end{proof}

We generalize Theorem~\ref{collection}(e):

\begin{corollary} \label {0cvikn3i4Apfsyuksukje}
Let $n \geq 4$ and $S \in T_n$.  If $\cotype(S) \geq n-2$, then $B_S$ is rigid.  
\end{corollary}
	
\begin{proof}
We have $| J(S) | = \cotype(S)\ge n-2$, so $\min( |J(S)|-1, n-3 )=n-3$. 
Let $J$ be any subset of $J(S)$  satisfying $|J| = \min( |J(S)|-1, n-3 ) = n-3$;
since $S \in T_n$, we have $S_J \in T_3$, so $B_{S_J}$ is rigid by Theorem~\ref{jeuyuBdgl3r6hj24dUi476dyb3i8}.
Then Proposition~\ref{kcj203jfs0dvj} implies that $B_S$ is rigid.
\end{proof}

\begin{example} \label {0v34i8IeEkrj6ssyfuwygrwejk37rhf}
Proposition \ref{kcj203jfs0dvj} can settle cases not covered by Corollary~\ref{0cvikn3i4Apfsyuksukje} or by any of the previous results. For instance:
\begin{quote} \it
If $a_1,\dots,a_m \in (\Nat\setminus\{0\})^m$ ($m\ge1$) satisfy $a_i \nmid \lcm(3,a_1, \dots \widehat{a_i} \dots, a_m)$ for all $i\in \{1, \dots, m\}$,
then $B_{a_1, \dots, a_m,3,3,3}$ is rigid.
\end{quote}
Indeed, let $S = (a_1, \dots, a_m,3,3,3)$, then $J(S) = \{ 1, \dots, m \}$ and  (by Theorem~\ref{collection}(c)) $B_{a_i,3,3,3}$ is rigid for each $i \in J(S)$,
so $B_S$ is rigid by Proposition~\ref{kcj203jfs0dvj}.
For instance, $B_{2,5,7,3,3,3}$ is rigid.
\end{example}

\begin{notation}  \label {cknp2o39tiogrnfsiiw}
Given $S = (a_1, \dots, a_n) \in (\Nat\setminus\{0\})^n$ ($n\ge2$) and a subset $M$ of $\{1, \dots, n\}$, define a positive integer
$\Delta_M(S)$ by $\Delta_\emptyset(S)=1$ and
$$
\Delta_M(S) = \frac{ \lcm(S) }{ \gcd \setspec{\lcm(S_j)}{ j \in M} } \qquad \text{if $M \neq \emptyset$}.
$$  
Observe that
$$
\Delta_M(S) = \Delta_{J(S) \cap M}(S) \quad \text{and} \quad \Delta_M(S) = 1 \Leftrightarrow M \cap J(S) = \emptyset ,
$$
because $\lcm(S_j)=\lcm(S)$ for each $j \in \{1, \dots, n\} \setminus J(S)$,
and $\lcm(S_j)$ is a proper divisor of $\lcm(S)$ for each $j \in J(S)$.
\end{notation}

\begin{proposition} \label {c98vhb2rtdhghsxjuej}
Let $n\ge4$, $S \in (\Nat\setminus\{0\})^n$ and $M = \setspec{ j \in J(S) }{ \text{$B_{S_j}$ is rigid} }$. 
\begin{enumerata}

\item $\Integ( \ker D ) \subseteq  \Delta_M(S) \cdot \Integ$\ \  for all $D \in \hlnd(B_S) \setminus \{0\}$.

\item $\Delta_M(S) =1 \iff M=\emptyset$

\end{enumerata}
\end{proposition}

\begin{proof}
We use the following notation:
$S=(a_1,\dots,a_n)$, $\overline S = (d_1, \dots\, d_n)$,
$x_1,\dots,x_n$ as in Definition~\ref{8vroq98vbcu56msw9nw0e},
$L = \lcm(S)$ and, for each $i \in \{1, \dots, n\}$, $L_i = \lcm(S_i)$ and $\delta_i = \gcd( d_1, \dots\, \widehat{d_i}\, \dots, d_n )$.
Then $\delta_i = \gcd( \frac{L}{a_1},  \dots\, \widehat{ \textstyle\frac{L}{a_i} } \, \dots, \frac{L}{a_n} )
= \frac{L}{L_i} \gcd( \frac{L_i}{a_1},  \dots\, \widehat{ \frac{L_i}{a_i} } \, \dots, \frac{L_i}{a_n} ) = \frac{L}{L_i}$
for each $i \in \{1, \dots, n\}$, so
$\lcm( \delta_j \, : \, j \in M )
= \lcm( \frac{L}{L_j} \, : \, j \in M )
= \frac{ L }{\gcd( L_j \, : \, j \in M ) } = \Delta_M(S)$ and consequently
$\bigcap_{j \in M} \delta_j \Integ = \Delta_M(S) \cdot \Integ$.
So, to prove (a), it suffices to show that
\begin{equation}  \label {ckjvbo239ehfso}
\text{for each $D \in \hlnd(B_S) \setminus \{0\}$ and each $j \in M$, \quad $\Integ( \ker D ) \subseteq  \delta_j \Integ$.}
\end{equation}
We prove this by contradiction: suppose that 
$D \in \hlnd(B_S) \setminus \{0\}$ and $j \in M$ satisfy $\Integ( \ker D ) \nsubseteq  \delta_j \Integ$.
By \ref{p0cfi2k309cbqp90ws}\eqref{02dj7edj9w34diey}, there exists an irreducible $D' \in \hlnd(B_S) \setminus \{0\}$
with $\ker(D')=\ker(D)$, so we may choose $D$ irreducible.
Then there exists a homogeneous element $h \in \ker(D)\setminus\{0\}$ such that $\deg(h) \notin \delta_j \Integ$.
We can write $h = \sum_{(i_1,\dots,i_n) \in E} \lambda_{i_1,\dots,i_n} x_1^{i_1} \cdots x_n^{i_n}$ 
for some finite set $E \subset \Nat^n$, where for each $(i_1,\dots,i_n) \in E$ we have 
$\lambda_{i_1,\dots,i_n} \in \bk^*$ and $\deg( x_1^{i_1} \cdots x_n^{i_n} ) = \deg(h) \notin \delta_j \Integ$. 
So for each $(i_1,\dots,i_n) \in E$ we have $\deg( x_j^{i_j} ) \notin \delta_j \Integ$ and in particular $i_j>0$.
This shows that $x_j \mid h$ in $B_S$. Since $\ker(D)$ is factorially closed in $B_S$, we get $x_j \in \ker D$.
Then $D$ induces a locally nilpotent derivation $\bar D$ of the ring $B_S/\langle x_j \rangle \isom B_{S_j}$,
and $\bar D \neq 0$ because $D( B_S ) \nsubseteq \langle x_j \rangle$ (since $D$ is irreducible).
So $B_{S_j}$ is not rigid, which contradicts $j \in M$.
This proves \eqref{ckjvbo239ehfso}, so assertion (a) is proved.
Part (b) follows from the observation (made in \ref{cknp2o39tiogrnfsiiw}) that, for any subset $M$ of $\{1, \dots, n\}$,
we have $\Delta_M(S) = 1$ if and only if $M \cap J(S) = \emptyset$.
\end{proof}

It immediately follows:

\begin{corollary} \label {FHgEg4FxRtei8cbql}
Let $n\ge4$ and $S \in (\Nat\setminus\{0\})^n$. 
If some $D \in \hlnd(B_S) \setminus \{0\}$ satisfies $\Integ( \ker D ) = \Integ$ then $B_{S_j}$ is non-rigid  for every $j \in J(S)$.
\end{corollary}

Here is another consequence of Proposition~\ref{c98vhb2rtdhghsxjuej}, valid for $n=4$:

\begin{corollary} \label {iuhjgFuj93yd8i2h8HGUeru39}
Let $S \in T_4$ be such that $\cotype(S)>0$ and define $\delta = \Delta_{ J(S) }(S) \in \Integ$. Then
$$
\text{$\delta>1$ \quad and \quad $\Integ( \ker D ) \subseteq \delta \Integ$ for all  $D \in \hlnd(B_S) \setminus \{0\}$.}
$$
\end{corollary}

\begin{proof}
We have $J(S) \neq \emptyset$ because $\cotype(S) > 0$.
If $j \in J(S)$ then $S_j \in T_3$, so $B_{S_j}$ is rigid by Theorem~\ref{jeuyuBdgl3r6hj24dUi476dyb3i8}.
So the set $M =  \setspec{ j \in J(S) }{ \text{$B_{S_j}$ is rigid} }$ of Proposition~\ref{c98vhb2rtdhghsxjuej}
is equal to $J(S)$ (which is not empty).
The desired conclusion follows from Proposition~\ref{c98vhb2rtdhghsxjuej}.
\end{proof}

\section{A remark about $\Proj(B_S)$}

As in the preceding section, we assume that $\bk$ is a field of characteristic zero.
	
	\begin{proposition}[Exercise 9.5 in \cite{Eisenbud}]   \label{projIsomorphism}
		Let $R = R_0 \oplus R_1 \oplus R_2 \oplus  \dots$  be an $\Nat$-graded Noetherian ring, let $d > 0$ and let $R^{(d)} = R_0 \oplus R_d \oplus R_{2d} \oplus \dots$. Then $\Proj R \isom \Proj R^{(d)}$.
	\end{proposition}

\begin{proposition} \label {isomorphicSubring}
Let $S = (a_1, \dots, a_n)$, $S' = (a_1', \dots, a_n') \in (\Nat\setminus\{0\})^n$ ($n\ge3$) and assume that $i \in \{1, \dots, n\}$
is such that $S \le^i S'$.  Then $B_S \isom B_{S'}^{(k)}$, where we define $k = a_i' / a_i \in \Nat \setminus\{0\}$.
Consequently,  $\Proj B_S \isom \Proj B_{S'}$.
\end{proposition}

\begin{proof}
We may assume that $i=1$.
Define $(d_1, \dots, d_n) = \overline{S}$ and $(d_1', \dots, d_n') = \overline{S'}$.
Let us prove:
\begin{equation}  \label {Vv9erRldvxjJhHcuiosX}
(d_1', \dots, d_n') = (d_1, k d_2, \dots, k d_n) \quad \text{and} \quad \gcd(d_1, k)=1.
\end{equation}
Let $L=\lcm(S)$,  $L'=\lcm(S')$, and $L_1 = \lcm(S_1) = \lcm(S_1')$. Let $g_1 = g_1(S')=g_1(S)$, i.e., $g_1 = \gcd(a_1', L_1) = \gcd(a_1, L_1)$.
We have $L = \lcm(a_1, L_1) = a_1 L_1/g_1$  and $L' = \lcm(a_1', L_1) = a_1' L_1/g_1$, so  for each $j \in \{1, \dots, n\}$, 
we have $d_j = L/a_j = \frac{a_1 L_1}{a_j g_1}$ and $d_j' = L'/a_j' = \frac{a_1' L_1}{a_j' g_1}$. This gives $d_1' = L_1/g_1 = d_1$
and, for $j\neq 1$, $d_j' =  \frac{a_1' L_1}{a_j' g_1} =  \frac{k a_1 L_1}{a_j g_1} = k d_j$;
this proves the first part of \eqref{Vv9erRldvxjJhHcuiosX}.
Since $\gcd(d_1,k)$ is a divisor of $\gcd(d_1, kd_2, \dots, kd_n) = \gcd(d_1', \dots, d_n') = 1$, \eqref{Vv9erRldvxjJhHcuiosX} is proved.
Now let $\Phi : \bk[X_1, \dots, X_n] \to \bk[Y_1, \dots, Y_n]$ be the $\bk$-homomorphism that sends $X_1$ to $Y_1^k$ and $X_j$ to $Y_j$ for $j>1$.
Noting that  $Y_1^{a_1'} + \cdots + Y_n^{a_n'} = Y_1^{ka_1} + Y_2^{a_2} + \cdots + Y_n^{a_n}$, we see that  
$\Phi^{-1} \big( \langle Y_1^{a_1'} + \cdots + Y_n^{a_n'} \rangle \big) =  \langle X_1^{a_1} + \cdots + X_n^{a_n} \rangle$,
so $\Phi$ induces an injective homomorphism 
$$
\phi : B_S = \bk[X_1, \dots, X_n] / \langle X_1^{a_1} + \cdots + X_n^{a_n} \rangle 
\to B_{S'} = \bk[Y_1, \dots, Y_n] / \langle Y_1^{a_1'} + \cdots + Y_n^{a_n'} \rangle .
$$
Then $B_S \isom \phi( B_S ) = \bk[y_1^k, y_2, \dots, y_n]$.
Since $\deg(y_j) = d_j'$ for all $j$, \eqref{Vv9erRldvxjJhHcuiosX} implies that $B_{S'}^{(k)} = \bk[y_1^k, y_2, \dots, y_n]$
and that $\phi$ is homogeneous (meaning $\phi \big( (B_S)_i \big) \subseteq (B_{S'})_{ki}$ for all $i \in \Nat$); so $B_S \isom B_{S'}^{(k)}$ as graded rings.
Then  $\Proj B_S \isom \Proj B_{S'}$ follows from Proposition~\ref{projIsomorphism}.
\end{proof}

\begin{remark}
It is interesting to observe that the geometry of $\Proj B_S$ is in general not sufficient to determine whether
or not $B_S$ is rigid (whereas the geometry of $\Spec B_S$ is of course sufficient).
For example, let $S = (2,3,3,2)$, $S'=(2,3,3,4)$ and $S'' = (10,3,3,4)$.
% Since $S = (2,3,3,2) <^4 (2,3,3,4) <^1 (10,3,3,4) = S'$,
Since $S <^4 S' <^1 S''$,
Proposition~\ref{isomorphicSubring} implies that $\Proj B_{S}  \isom \Proj B_{S'} \isom \Proj B_{S''}$.
However, $B_{S} = B_{2,3,3,2}$ is not rigid because $(2,3,3,2) \notin T_4$, and $B_{S''} = B_{10,3,3,4}$ is rigid by Theorem~\ref{collection}(e).
Interestingly, we don't know whether $B_{S'}$ is rigid or not.
\end{remark}

	\section{Stable Rigidity}

		\begin{lemma}  \label {pv923oF5O129rugfhg29}
			Let $E \subseteq F$ be a field extension, $m\ge1$, $t_1, \dots, t_m$ independent indeterminates over $F$ and
			$f,g \in E[t_1,\dots,t_m] \subseteq F[t_1,\dots,t_m]$.
			Then $f,g$ are relatively prime in $E[t_1,\dots,t_m]$ if and only if they are relatively prime in $F[t_1,\dots,t_m]$.
		\end{lemma}
		
\begin{comment}  %%%%%%%%%%%%%%%%%%%%%%%%%%%%%%%%%%%%%%%%%%%%%%%%%%%%%%%%%%%%%%%%%%%%%%%%%%%%%%%%%%%%%%%%%%%%%%%%%%%%%%%%%
%%%%%%%%%%%%%%%%%%%%%%%%%%%%%%%%%%%%%%%%%%%%%%%%%%%%%%%%%%%%%%%%%%%%%%%%%%%%%%%%%%%%%%%%%%%%%%%%%%%%%%%%%
%%%%%%%%%%%%%%%%%%%%%%%%%%%%%%%%%%%%%%%%%%%%%%%%%%%%%%%%%%%%%%%%%%%%%%%%%%%%%%%%%%%%%%%%%%%%%%%%%%%%%%%%%

		\begin{proof}
			We have $f,g \in A = E[t_1,\dots,t_m] \subseteq B = F[t_1,\dots,t_m]$.
			Assume that $f,g$ are not relatively prime in $B$.  
			Then there exists  $h \in B \setminus F$ such that $h \mid f$ and $h \mid g$ in $B$.
			We have $\deg_{t_i}(h)>0$ for some $i$; consider
			$$
			f,g \in A' = E(t_1,\dots,t_{i-1},t_{i+1}, \dots, t_{m})[t_i] \subseteq B' = F(t_1,\dots,t_{i-1},t_{i+1}, \dots, t_{m})[t_i] .
			$$
			Let $G = \gcd_{B'}(f,g)$, computed by the Euclidean algorithm in $B'$.
			Since $h \mid f$ and $h \mid g$ in $B'$, we have $h \mid G$ in $B'$, so $\deg_{t_i}(G)>0$.
If we compute the gcd of $f$ and $g$ in $A'$ by using the Euclidian algorithm, then of course we obtain the same $G$; so $G = \gcd_{A'}(f,g)$.
Since  $\deg_{t_i}(G)>0$, we see that $f,g$ are not relatively prime in $A'$; it follows that $f,g$ are not relatively prime in $A$.
			This shows that if $f,g$ are relatively prime in $A$ then they are relatively prime in $B$.
			The converse is trivial.
		\end{proof}
		
\end{comment}  %%%%%%%%%%%%%%%%%%%%%%%%%%%%%%%%%%%%%%%%%%%%%%%%%%%%%%%%%%%%%%%%%%%%%%%%%%%%%%%%%%%%%%%%%%%%%%%%%%%%%%%%%
%%%%%%%%%%%%%%%%%%%%%%%%%%%%%%%%%%%%%%%%%%%%%%%%%%%%%%%%%%%%%%%%%%%%%%%%%%%%%%%%%%%%%%%%%%%%%%%%%%%%%%%%%
%%%%%%%%%%%%%%%%%%%%%%%%%%%%%%%%%%%%%%%%%%%%%%%%%%%%%%%%%%%%%%%%%%%%%%%%%%%%%%%%%%%%%%%%%%%%%%%%%%%%%%%%%

Verification of the above Lemma is left to the reader. We quote Theorem 3.1 of \cite{DeBondt_MasonThm_2009}:
		
		\begin{theorem}  \label {pc09vi2b39efqkw}
			Let $m\ge1$, $n\ge3$, $g_1, \dots, g_n \in \Comp[t_1,\dots,t_m] = \Comp^{[m]}$ and $a_1, \dots, a_n \in \Nat\setminus\{0\}$ be such that:
			\begin{itemize}
				
				\item $g_1^{a_1} + \cdots + g_n^{a_n} = 0$
				
				\item  $g_1, \dots, g_n$ are not all constant
				
				\item given any $1 \le i_1 < \dots < i_s \le n$ such that $g_{i_1}^{a_{i_1}} + \cdots + g_{i_s}^{a_{i_s}} = 0$, 
				we have $\gcd( g_{i_1}, \dots, g_{i_s}) = 1$.
				
			\end{itemize}
			Then $\sum_{i=1}^n \frac1{a_i} > \frac1{d-1}$ where $d$ is the dimension of the $\Comp$-vector space spanned by $g_1^{a_1},  \dots,  g_n^{a_n}$.
		\end{theorem}

		\begin{corollary}   \label {0f9b23td8awpe7g}
			Let $K$ be a field of characteristic zero, $m\ge1$, $n\ge3$, $g_1, \dots, g_n \in K[t_1,\dots,t_m] = K^{[m]}$
			and $a_1, \dots, a_n \in \Nat\setminus\{0\}$ be such that:
			\begin{itemize}
				
				\item[(i)] $g_1^{a_1} + \cdots + g_n^{a_n} = 0$
				
				\item[(ii)] $\sum_{i=1}^n \frac1{a_i} \le \frac1{n-2}$
				
				\item[(iii)]  $g_1, \dots, g_n$ are pairwise relatively prime.
				
			\end{itemize}
			Then $g_1, \dots, g_n \in K$.
		\end{corollary}
		
		\begin{proof}
First consider the case where $K=\Comp$. Let $d$ be the dimension of the $\Comp$-vector space spanned by $g_1^{a_1},  \dots,  g_n^{a_n}$.
Note that $d<n$ by (i);
if $d\le1$ then the conclusion ($g_1, \dots, g_n \in K$) immediately follows from (iii), so we may assume that $1 < d < n$.
Then $\frac1{d-1}$ is defined and $\frac1{d-1}\ge \frac1{n-2}$. So (ii) gives  $\sum_{i=1}^n \frac1{a_i} \le \frac1{d-1}$
and Theorem~\ref{pc09vi2b39efqkw} implies that $g_1, \dots, g_n \in K$.
So the case $K=\Comp$ of Corollary \ref{0f9b23td8awpe7g} is true.

			Now let $K$ be arbitrary. Let $K_0 \subseteq K$ be the extension of $\Rat$ generated by the coefficients of $g_1, \dots, g_n$.
			Then $K_0$ can be embedded in $\Comp$;
			more precisely, if we choose a sufficiently large overfield $L$ of $K$ then we may find a copy of $\Comp$ in $L$
			forming a diagram of fields as in the left part of:
			$$
			
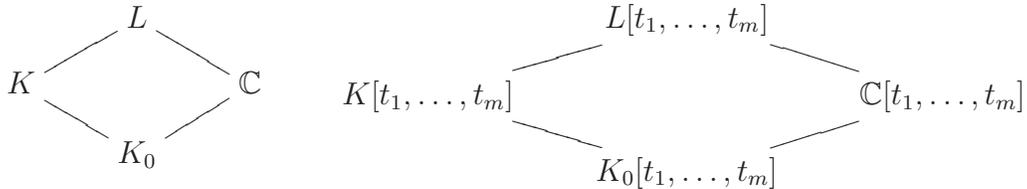
\xymatrix@R=10pt{
				& L  \\
				K \ar@{-}[ru] \ar@{-}[rd] && \Comp \ar@{-}[lu] \ar@{-}[ld]  \\
				& K_0
			}
			\qquad
			\xymatrix@R=10pt{
				& L[t_1,\dots,t_m]  \\
				K[t_1,\dots,t_m] \ar@{-}[ru] \ar@{-}[rd] && \Comp[t_1,\dots,t_m] \ar@{-}[lu] \ar@{-}[ld]  \\
				& K_0[t_1,\dots,t_m]
			}
			$$
			Now $g_1, \dots, g_n \in K_0[t_1,\dots,t_m] \subseteq \Comp[t_1,\dots,t_m]$ are  pairwise relatively prime in $\Comp[t_1,\dots,t_m]$
			by Lemma \ref{pv923oF5O129rugfhg29}.
			So $g_1, \dots, g_n$ satisfy (i--iii) as elements of $\Comp[t_1,\dots,t_m]$.
			By the case $K=\Comp$ of Corollary~\ref{0f9b23td8awpe7g}, we get $g_1, \dots, g_n \in \Comp$, so $g_1, \dots, g_n \in K$.
		\end{proof}
		
We need the notion of relatively prime elements of an arbitrary domain.

\begin{definition} \label {cp0g32fd6dsx0hXbc3t8}
Let $x,y \in B$ where $B$ is a domain.
One says that $x,y$ are {\it relatively prime} in $B$ if the following hold:
\begin{itemize}
\item[(i)] $xB \cap yB = xyB$ \qquad (ii) if $0 \in \{x,y\}$ then $\{x,y\} \cap B^* \neq \emptyset$.
\end{itemize}
Note that this agrees with the usual notion when $B$ is a UFD.
\end{definition}
		
		\begin{remark}  \label {9029bf923i0d}
			Let $x,y$ be  relatively prime elements of a domain $B$.
			\begin{enumerata}
				
				\item If $S$ is a multiplicative set of $B$ such that $0 \notin S$, then $x,y$ are  relatively prime in $S^{-1}B$. 
				
				\item If $A$ is a factorially closed subring of $B$ such that $x,y \in A$, then $x,y$ are  relatively prime in $A$.
				
				\item If $B' = B^{[N]}$ for some $N\ge0$ then $x,y$ are relatively prime in $B'$.
				
			\end{enumerata}
		\end{remark}

\begin{definition}
Let $B$ be a domain of characteristic zero.
\begin{enumerata}

\item $\ML(B) = \bigcap_{ D \in \lnd(B) } \ker(D)$

\item The \textit{rigid core} of $B$ is defined as $\bigcap_{i=0}^\infty \ML_i(B)$,
where one defines $\ML_0(B)=B$ and $\ML_{i+1}(B) = \ML( \ML_i B )$ for all $i$.

\end{enumerata}
\end{definition}
		
		The next result generalizes Theorem 6.1(a) of \cite{freudenburg2013}.

		\begin{theorem}  \label {cp09vvh2930cvqn3909}
			Let $B$ be a domain of characteristic zero,
			$n\ge3$, $x_1, \dots, x_n \in B$ and $a_1,\dots,a_n \in \Nat\setminus\{0\}$.  Assume that
			\begin{itemize}
				
				\item[(i)] $x_1^{a_1} + \cdots + x_n^{a_n} = 0$
				
				\item[(ii)] $\sum_{i=1}^n \frac1{a_i} \le \frac1{n-2}$
				
				\item[(iii)]  $x_1, \dots, x_n$ are pairwise relatively prime in $B$.
				
			\end{itemize}
			Then  $x_1,\dots,x_n \in \ML(R)$ in each of the following cases:
\begin{enumerata}

\item $R$ is a factorially closed subring of $B$ satisfying $x_1,\dots,x_n \in R$;

\item $R = B^{[N]}$ for some $N$. 

\end{enumerata}
Moreover, $x_1, \dots, x_n$ belong to the rigid core of $B$.
\end{theorem}
		
		\begin{proof}
(a) Let $R$ be a factorially closed subring of $B$ satisfying $x_1,\dots,x_n \in R$.
Let $D \in \lnd(R)$ and let $A = \ker(D)$; we claim that $x_1,\dots,x_n \in A$.
To show this, we may assume that $D\neq0$. Let $S = A \setminus \{0\}$, then $S^{-1}R = K^{[1]}$ where we set $K = \Frac(A)$.
Note that $x_1,\dots,x_n$ are pairwise relatively prime in $R$ by Remark \ref{9029bf923i0d}(b),
so they are pairwise relatively prime in  $S^{-1}R = K^{[1]}$ by  Remark \ref{9029bf923i0d}(a).
Then $x_1,\dots,x_n \in K$ by Corollary \ref{0f9b23td8awpe7g}. As $A$ is factorially closed in $R$,
we have $R \cap K = A$ and hence $x_1,\dots,x_n \in A$.
This argument shows that $x_1,\dots,x_n \in \ML(R)$, so case (a) is proved.

(b) Assume that $R = B^{[N]}$ for some $N$.
Then $x_1,\dots,x_n$ are pairwise relatively prime in $R$ by Remark \ref{9029bf923i0d}(c).
Applying part (a) to $x_1,\dots,x_n \in R$ shows that $x_1,\dots,x_n \in \ML(R')$ for any factorially closed subring $R'$ of $R$ containing $x_1, \dots, x_n$;
in particular,  $x_1,\dots,x_n \in \ML(R)$. So we are done in case (b).

Observe that  $x_1,\dots,x_n \in \ML_0(B)$ and that if $i$ is such that  $x_1,\dots,x_n \in \ML_i(B)$ then
$x_1,\dots,x_n \in \ML_{i+1}(B)$  ($R=\ML_i(B)$ is a factorially closed subring of $B$,
so case (a) gives $x_1,\dots,x_n \in \ML(R)$).
 So $x_1,\dots,x_n$ belong to the rigid core of $B$.
\end{proof}

One says that a ring $B$  of characteristic zero is \textit{stably rigid} if $B \subseteq \ML(R)$ for every overring $R$ of $B$
such that $R = B^{[N]}$ for some $N$.
In the following statement, we set $B_{a_1,\dots,a_n} = \bk[X_1,\dots,X_n]/ \lb X_1^{a_1} + \cdots + X_n^{a_n} \rb$
where $\bk$ is a field of characteristic zero.
		
\begin{corollary}  \label {ck9fweymrdkayecumcff83}
Let $n\ge3$ and $(a_1,\dots,a_n) \in (\Nat\setminus\{0\})^n$ be such that $\sum_{i=1}^n \frac1{a_i} \le \frac1{n-2}$.
Then $B_{a_1,\dots,a_n}$ is stably rigid.
\end{corollary}
		
\begin{proof}
The case $n=3$ is known (Theorem 7.1(b) of \cite{freudenburg2013}), so we may assume that $n \geq 4$.
Write $B = B_{a_1,\dots,a_n} = \bk[x_1,\dots,x_n]$ where $x_i$ is the canonical image of $X_i$ in $B$. 
By Lemma~\ref{pc04b98wjwk77subrq9a}, $x_1,\dots,x_n$ are prime elements of $B$ and are pairwise relatively prime in $B$.
Consider an overring  $R = B^{[N]}$ of $B$ for some $N$.
Then $x_1,\dots,x_n \in \ML(R)$ by case (b) of Theorem~\ref{cp09vvh2930cvqn3909}, so $B$ is stably rigid.
\end{proof}

\bibliographystyle{plain}

\end{document}